\documentclass[12pt]{amsart}

\usepackage{ucs}

\usepackage{amssymb}
\usepackage{amsthm}
\usepackage{amsmath}
\usepackage{latexsym}
\usepackage[cp1251]{inputenc}
\usepackage{graphicx}
\usepackage{wrapfig}
\usepackage{caption}
\usepackage{subcaption}
\usepackage{indentfirst}
\usepackage[left=2.6cm,right=2.6cm,top=2.6cm,bottom=2.6cm,bindingoffset=0cm]{geometry}
\usepackage{enumerate}
\usepackage{makecell}
\usepackage{float}

\DeclareMathOperator{\aut}{Aut}

\DeclareMathOperator{\cay}{Cay}
\DeclareMathOperator{\cyc}{Cyc}
\DeclareMathOperator{\dev}{dev}

\DeclareMathOperator{\GL}{GL}
\DeclareMathOperator{\id}{id}

\DeclareMathOperator{\iso}{Iso}

\DeclareMathOperator{\sym}{Sym}

\DeclareMathOperator{\alg}{Alg}
\DeclareMathOperator{\dimwl}{dim_{WL}}

\DeclareMathOperator{\WL}{WL}

\DeclareMathOperator{\GCD}{GCD}

\DeclareMathOperator{\ind}{ind}

\DeclareMathOperator{\Hol}{Hol}

\def\r{\mathrm{right}}

\def\tm#1{\item[{\rm (#1)}]}

\makeatletter 
\def\@seccntformat#1{\csname the#1\endcsname. } 
\def\@biblabel#1{#1.}

\makeatother

\title{On a family of divisible design digraphs}

\author{Mikhail Muzychuk}
\address{Ben Gurion University of the Negev, Beer Sheva, Israel}
\email{muzychuk@bgu.ac.il}

\author{Grigory Ryabov}
\address{School of Mathematical Sciences, Hebei Key Laboratory of Computational Mathematics and Applications, Hebei Normal University, Shijiazhuang 050024, P. R. China}
\address{Ben Gurion University of the Negev, Beer Sheva, Israel}
\email{gric2ryabov@gmail.com}

\thanks{The second author was supported by the grant of The Natural Science Foundation of Hebei Province (project No.~A2023205045) and the grant of The Israel Science Foundation (project No.~87792731)}

\date{}

\newtheorem{prop}{Proposition}[section]

\newtheorem{lemm}[prop]{Lemma}
\newtheorem{theo}[prop]{Theorem}

\newtheorem{corl}[prop]{Corollary}

\theoremstyle{definition}

\newtheorem*{rem1}{Remark~1}
\newtheorem*{rem2}{Remark~2}

\begin{document}

\maketitle

\begin{abstract}
For every odd prime power~$q$, a family of pairwise nonisomorphic normal arc-transitive divisible design Cayley digraphs with isomorphic neighborhood designs over a Heisenberg group of order~$q^3$ is constructed. It is proved that these digraphs are not distinguished by the Weisfeiler-Leman algorithm and have the Weisfeiler-Leman dimension~$3$.
\end{abstract}

\maketitle

\section{Introduction}

Let $\Gamma=(\Omega,D)$ be a digraph (possibly, with loops). We say that a vertex $\alpha\in \Omega$ \emph{dominates} a vertex $\beta\in \Omega$ if $(\alpha,\beta)$ is an arc of $\Gamma$. The digraph $\Gamma$ is said to be \emph{regular} of degree~$k$ if each vertex of $\Gamma$ dominates exactly $k$ vertices and is dominated by exactly $k$ vertices.  The digraph $\Gamma$ is said to be \emph{asymmetric} if $(\beta,\alpha)\notin D$ for all $\alpha,\beta\in \Omega$ such that $\alpha\neq \beta$ and $(\alpha,\beta)\in D$.

A $k$-regular asymmetric digraph on~$v$ vertices is called a \emph{divisible design digraph} (\emph{DDD} for short) with parameters $(v,k,\lambda_1,\lambda_2,m,n)$ if its vertex set can be partitioned into~$m$ classes of size~$n$ such that for any two distinct vertices $\alpha$ and $\beta$, the number of vertices that dominates or being dominated by both $\alpha$ and $\beta$ is equal to $\lambda_1$ if $\alpha$ and $\beta$ belong to the same class and~$\lambda_2$ otherwise. The notion of a divisible design graph was introduced in~\cite{HKM} and extended to digraphs in~\cite{CK}. Several constructions of DDDs can be found in~\cite{CH,CK,HKM,Kabanov,PS}.

DDDs are closely related to symmetric (group) divisible designs (see~\cite{BJL,ScSp} for the definition). More precisely, if $\Gamma$ is a DDD, then an incidence structure whose points are vertices and blocks are out-neighborhoods of vertices of $\Gamma$ is a symmetric divisible design called the \emph{neighborhood design} of $\Gamma$ and denoted by $\mathcal{D}=\mathcal{D}(\Gamma)$. An adjacency matrix of $\Gamma$ coincides with an incidence matrix of $\mathcal{D}$. Observe that even if DDDs $\Gamma_1$ and $\Gamma_2$ are not isomorphic, the neighborhood designs $\mathcal{D}(\Gamma_1)$ and $\mathcal{D}(\Gamma_2)$ can be isomorphic. For example, if $\Gamma_2$ is obtained from $\Gamma_1$ by a dual Seidel switching, then $\mathcal{D}(\Gamma_1)$ and $\mathcal{D}(\Gamma_2)$ are isomorphic, whereas $\Gamma_1$ and $\Gamma_2$ can be nonisomorphic (see~\cite{GHKS}). An example of nonisomorphic DDDs with isomorphic neighborhood designs can be found in~\cite{HKM}. It seems interesting to ask how many pairwise nonisomorphic DDDs may have the same (up to isomorphism) neighborhood design.

In fact, the latter question is a main motivation of this paper. We construct a family of pairwise nonisomorphic Cayley DDDs $\Gamma_i$, $i\in I$, over a Heisenberg group with the same neighborhood design that have such additional properties like normality and arc-transitivity. By a \emph{Cayley digraph} $\Gamma=\cay(G,X)$ over a finite group $G$ with a non-empty connection set $X\subseteq G$, we mean the digraph with vertex set $G$ and arc set $D=\{(g,xg):~x\in X,~g\in G\}$. The automorphism group $\aut(\Gamma)$ contains a regular subgroup $G_\r$ consisting of all right translations of $G$ and $\Gamma$ is said to be \emph{normal} if $G_\r$ is normal in $\aut(\Gamma)$.

The Cayley digraph $\Gamma$ is a DDD with parameters $(v,k,\lambda_1,\lambda_2,m,n)$ if and only if the connection set $X$ of $\Gamma$ is a symmetric \emph{divisible difference set} (\emph{DDS} for short, see~\cite{AJP} for the definition) with parameters $(m,n,k,\lambda_1,\lambda_2)$. In the latter case, the neighborhood design $\mathcal{D}(\Gamma)$ is called a \emph{development} of~$X$ and denoted by~$\dev(X)$. The point and block sets of $\dev(X)$ are $G$ and $\{Xg:~g\in G\}$, respectively. More details on DDSs and Cayley DDDs can be found in~\cite{AJP,BJL} and in~\cite{CKS,CS,KS}, respectively. Our construction of Cayley DDDs is based on a recent construction of DDSs~\cite{MR}.

To check that $\Gamma_i$'s are pairwise nonisomorphic, we compute isomorphisms of a special $S$-ring (Section~$3$) whose automorphism group coincides with an automorphism group of each $\Gamma_i$. This also allows us to establish normality and arc-transitivity of $\Gamma_i$'s. Besides, a knowledge of the above isomorphisms implies some more interesting properties of $\Gamma_i$'s which are discussed further. Clearly, $\Gamma_i$'s have the same parameters because they have the same neighborhood design. However, they possess a stronger similarity from a combinatorial point of view. Namely, the tensors constructed from $\Gamma_i$'s by the Weisfeiler-Leman algorithm~\cite{WeisL} are the same, or, in other words, all $\Gamma_i$'s are pairwise \emph{WL-equivalent}. Nevertheless, each $\Gamma_i$ can be distinguished from any other graph by the tensor constructed from $\Gamma_i$ by the $3$-dimensional Weisfeiler-Leman algorithm, i.e. the \emph{Weisfeiler-Leman dimension} (\emph{WL-dimension} for short) $\dimwl(\Gamma_i)$ of $\Gamma_i$ is at most~$3$. In particular, the isomorphism between $\Gamma_i$ and any other graph can be verified in polynomial time. For a background and details on the Weisfeiler-Leman algorithm and dimension, we refer the readers to~\cite{CP,FKV,Grohe,Weis}.

Given a power~$q$ of a prime~$p$, put $\log(q)=\log_p(q)$. A Heisenberg group (of dimension~$3$) over a finite field $\mathbb{F}_q$ of order~$q$ is denoted by $H_3(q)$. The Euler function of a positive integer~$n$ is denoted by $\phi(n)$. To sum up the whole previous discussion, we formulate the theorem below which is a main result of the paper.

\begin{theo}\label{main}
Let $q$ be an odd prime power. There exist $I\subseteq \mathbb{F}_q$ and a family $\Gamma_i$, $i\in I$, of pairwise nonisomorphic normal arc-transitive divisible design Cayley digraphs with parameters~$(q^3,q^2,0,q,q^2,q)$ over $H_3(q)$ such that:
\begin{enumerate}

\tm{1} $|I|\geq \left[\frac{\phi(q+1)}{2\log(q)}\right]$,

\tm{2} the neighborhood designs $\mathcal{D}(\Gamma_i)$ and $\mathcal{D}(\Gamma_j)$ are isomorphic,

\tm{3} $\Gamma_i$ and $\Gamma_j$ are $\WL$-equivalent,

\tm{4} $\dimwl(\Gamma_i)\leq 3$
 
\end{enumerate}
for all $i,j\in I$.
\end{theo}

\begin{rem1}
By the estimates for the Euler function (see, e.g.,~\cite[Theorem~15]{RS}), we have
$$\frac{\phi(q+1)}{2\log(q)}\geq \frac{q+1}{2c\log(q)\log(\log(q+1))} \underset{q \to \infty}{\longrightarrow} \infty$$
for some constant $c>0$.
\end{rem1}

\begin{rem2}
The WL-dimension of the digraphs from Theorem~\ref{main} is equal to~$3$ if $q>9$ and equal to~$2$ if $q=3$ (see Proposition~\ref{wldim3}). It would be interesting to determine exact value of the WL-dimension for these graphs in cases $q=5$ and $q=9$. 
\end{rem2}

We finish the introduction with a brief outline of the paper. Section~$2$ contains a necessary background of coherent configurations, $S$-rings (Schur rings), and WL-dimension. In Section~$3$, we study an $S$-ring over a Heisenberg group and its isomorphisms that will be used for the proof of Theorem~\ref{main}. Finally, we prove Theorem~\ref{main} in Section~$3$. We use the standard group-theoretical notations in the text. Namely, if $G$ is a group and $H\leq G$, then the centralizer and normalizer of $H$ in $G$ are denoted by $C_G(H)$ and $N_G(H)$, respectively, whereas the center of $G$ is denoted by $Z(G)$. The identity element of $G$ and the set of all nontrivial elements of $G$ are denoted by~$e$ and $G^\#$, respectively. The symmetric group of the set $\Omega$ is denoted by $\sym(\Omega)$. If $\Delta\subseteq \Omega$ is invariant under $f\in \sym(\Omega)$, then the bijection induced by $f$ on $\Delta$ is denoted by $f^{\Delta}$. The holomorph of a group $G$ is denoted by $\Hol(G)$.

\section{Preliminaries}

Throughout the text, we freely use basic definitions and facts from the theories of coherent configurations and $S$-rings (Schur rings) and recall in this section some of them which will be crucial for further explanation. For a background of coherent configurations and $S$-rings, we refer the readers to~\cite{CP} and~\cite{Ry}, respectively, where the terminology used in the paper and all facts not explained in detail are contained.

\subsection{Coherent configurations}

Let $\mathcal{X}=(\Omega,S)$ be a coherent configuration on a finite set $\Omega$ and $S$ the set of all basis relations of $\mathcal{X}$. If the diagonal of $\Omega\times \Omega$ is a basis relation of $\mathcal{X}$, then $\mathcal{X}$ is an (\emph{association}) \emph{scheme}. A binary relation on $\Omega$ is defined to be a \emph{relation} of $\mathcal{X}$ if it is a union of some basis relations. Given a coherent configuration $\mathcal{X}^\prime$ on $\Omega$, put $\mathcal{X}^\prime\geq \mathcal{X}$ if every basis relation of $\mathcal{X}$ is a relation $\mathcal{X}^\prime$.

If $s\in S$ and $\alpha\in \Omega$, then $n_s$ denotes the valency of $S$ and $\alpha s=\{\beta\in \Omega:~(\alpha,\beta)\in s\}$. A scheme $\mathcal{X}$ such that $n_s=1$ for all $s\in S$ is said to be \emph{regular}. The set of all fibers of $\mathcal{X}$ is denoted by $F(\mathcal{X})$. If $\Delta\in F(\mathcal{X})$, then the pair 
$$\mathcal{X}_{\Delta}=(\Delta,S_{\Delta}),$$
where $S_{\Delta}=\{s\in S:~s\subseteq\Delta\times\Delta\}$, is a coherent configuration on $\Delta$. Given $r,s,t\in S$, the corresponding intersection number of $\mathcal{X}$ is denoted by $c_{rs}^t$. 

Let $\mathcal{X}=(\Omega,S)$ and $\mathcal{X}^{\prime}=(\Omega^{\prime},S^{\prime})$ be coherent configurations. An \emph{algebraic isomorphism} from $\mathcal{X}$ to $\mathcal{X}^{\prime}$ is defined to be a bijection $\varphi:S\rightarrow S^{\prime}$ such that
$$c_{rs}^t=c_{\varphi(r),\varphi(s)}^{\varphi(t)}$$
for all $r,s,t\in S$. Every algebraic isomorphism can be extended in a natural way to a bijection between the relations of $\mathcal{X}$ and $\mathcal{X}^\prime$. The set of all algebraic isomorphisms from $\mathcal{X}$ to itself (algebraic automorphisms) forms the subgroup of $\sym(S)$ denoted by $\aut_{\alg}(\mathcal{X})$.

A (\emph{combinatorial}) \emph{isomorphism} from $\mathcal{X}$ to $\mathcal{X}^{\prime}$ is defined to be a bijection $f:\Omega\rightarrow \Omega^{\prime}$ such that $S^{\prime}=f(S)$, where $f(S)=\{f(s):~s\in S\}$ and $f(s)=\{(f(\alpha),f(\beta)):~(\alpha,~\beta)\in s\}$. The group $\iso(\mathcal{X})$ of all isomorphisms from $\mathcal{X}$ onto itself has a normal subgroup
$$\aut(\mathcal{X})=\{f\in \iso(\mathcal{X}): f(s)=s~\text{for every}~s\in S\}$$
called the \emph{automorphism group} of $\mathcal{X}$. A coherent configuration is said to be \emph{schurian} if $S$ is equal to the set of all orbits of $\aut(\mathcal{X})$ acting on $\Omega^2$ componentwise. A schurian coherent configuration is said to be \emph{$2$-minimal} if $\aut(\mathcal{X})$ does not have a proper subgroup with the same set $S$ of orbits in its componentwise action on $\Omega^2$.

Every isomorphism of coherent configurations induces in a natural way the algebraic isomorphism of them. However, not every algebraic isomorphism is induced by a combinatorial one. A coherent configuration is said to be \emph{separable} if every algebraic isomorphism from it to any coherent configuration is induced by an isomorphism. The set of all algebraic automorphisms of $\mathcal{X}$ induced by isomorphisms forms the subgroup of $\aut_{\alg}(\mathcal{X})$ denoted by $\aut^{\ind}_{\alg}(\mathcal{X})$.

\begin{lemm}\cite[Theorem~2.3.33]{CP}\label{regularsep}
Every regular scheme is separable.
\end{lemm}

\begin{lemm}\cite[Lemma~9.2]{EP}\label{deletefiber}
Let $\mathcal{X}=(\Omega,S)$ be a coherent configuration and $\Delta\in F(\mathcal{X})$ such that for every $\Delta\neq \Delta^\prime\in F(\mathcal{X})$ there is a basis relation $r\subseteq \Delta\times \Delta^\prime$ of valency~$1$. Then 
\begin{enumerate}

\tm{1} $\aut(\mathcal{X}_{\Delta})\cong \aut(\mathcal{X})$;

\tm{2} $\mathcal{X}$ is separable (schurian) whenever $\mathcal{X}_{\Delta}$ so is.
\end{enumerate}

\end{lemm}

Given $\alpha \in \Omega$, the one-point extension of $\mathcal{X}$ with respect to $\alpha$, i.e. the smallest coherent configuration $\mathcal{X}^\prime$ on $\Omega$ such that $\mathcal{X}^\prime \geq \mathcal{X}$ and $\{\alpha\}\in F(\mathcal{X}^\prime)$, is denoted by $\mathcal{X}_{\alpha}$. By~\cite[Proposition~3.3.3(1)]{CP},
\begin{equation}\label{autpoint}
\aut(\mathcal{X}_{\alpha})=\aut(\mathcal{X})_{\alpha}.
\end{equation}

The first statement of the lemma below is~\cite[Lemma~3.3.5(2)]{CP} whereas the second one is~\cite[Theorem~3.3.7(1)]{CP}. 

\begin{lemm}\label{onepoint}
Let $\mathcal{X}=(\Omega,S)$ be a coherent configuration and $\alpha\in \Omega$. Then the following statements hold.
\begin{enumerate}

\tm{1} For all $r,s,t\in S$, the binary relation $r\cap (\alpha s\times \alpha t)$ is a relation of $\mathcal{X}_{\alpha}$. 

\tm{2} If $\mathcal{X}$ is schurian, then $F(\mathcal{X}_{\alpha})=\{\alpha s:~s\in S\}$. 

\end{enumerate}
\end{lemm}

\subsection{$S$-rings}

Let $\mathcal{A}$ be an $S$-ring over a group $G$. The set of all basic sets of $\mathcal{A}$ is denoted by $\mathcal{S}=\mathcal{S}(\mathcal{A})$. Given $X,Y,Z\in \mathcal{S}(\mathcal{A})$, the corresponding structure constant of $\mathcal{A}$ is denoted by $c_{XY}^Z$. Due to~\cite[Eq.~2.1.14]{CP},
\begin{equation}\label{triangle}
|Z|c^{Z^{(-1)}}_{XY}=|X|c^{X^{(-1)}}_{YZ}=|Y|c^{Y^{(-1)}}_{ZX}
\end{equation}
for all $X,Y,Z\in \mathcal{S}(\mathcal{A})$. If $X\subseteq G$ and $\underline{X}\in \mathcal{A}$, then $X$ is defined to be an \emph{$\mathcal{A}$-set}. A subgroup of $G$ which is an $\mathcal{A}$-set is called an \emph{$\mathcal{A}$-subgroup}.

\begin{lemm}\cite[Proposition~22.1]{Wi}\label{sw}
Let $\mathcal{A}$ be an $S$-ring over $G$, $\xi=\sum \limits_{g\in G} c_g g\in \mathcal{A}$, where $c_g\in \mathbb{Z}$, and $c\in \mathbb{Z}$. Then $\{g\in G:~c_g=c\}$ is an $\mathcal{A}$-set.
\end{lemm}

If $\mathcal{A}$ is an $S$-ring over $G$, then $\mathcal{X}=\mathcal{X}(\mathcal{A})=(G,S)$, where $S=S(\mathcal{A})=\{r(X):~X\in \mathcal{S}(\mathcal{A})\}$ and $r(X)=\{(g,xg):~x\in X,~g\in G\}$, is a Cayley scheme over $G$, i.e. a scheme on the set $G$ such that $\aut(\mathcal{X})\geq G_\r$ (see~\cite[Section~2.4]{CP}). Due to~\cite[Theorem~2.4.16]{CP}, the mapping 
$$\mathcal{A}\mapsto \mathcal{X}(\mathcal{A})$$ 
is a partial order isomorphism between the $S$-rings and Cayley schemes over $G$. One can see that $\mathcal{A}_1\leq \mathcal{A}_2$ if and only if $\mathcal{X}(\mathcal{A}_1)\leq \mathcal{X}(\mathcal{A}_2)$.

Put 
$$\iso(\mathcal{A})=\iso(\mathcal{X}(\mathcal{A}))~\text{and}~\aut(\mathcal{A})=\aut(\mathcal{X}(\mathcal{A})).$$
Clearly, $\aut(\mathcal{A})\geq G_\r$. By the definitions, if $f\in \aut(\mathcal{A})$, $X\in \mathcal{S}(\mathcal{A})$, and $g\in G$, then
\begin{equation}\label{autring}
f(Xg)=Xf(g).
\end{equation}
The $S$-ring $\mathcal{A}$ is said to be \emph{normal} if $G_\r$ is normal in $\aut(\mathcal{A})$ or, equivalently, $\aut(\mathcal{A})$ is a subgroup of $\Hol(G)=G_\r\rtimes\aut(G)$. If $H$ is an $\mathcal{A}$-subgroup of $G$, then the set of all right and the set of all left $H$-cosets form imprimitivity systems of $\aut(\mathcal{A})$. 

The $S$-ring $\mathcal{A}$ is said to be \emph{separable} (\emph{schurian}, \emph{2-minimal}, respectively) if the Cayley scheme $\mathcal{X}(\mathcal{A})$ so is. If $\mathcal{A}$ is schurian, then $\mathcal{S}(\mathcal{A})$ is equal to the set of all orbits of the stabilizer $\aut(\mathcal{A})_e$ on $G$ and
\begin{equation}\label{schurnorm}
\iso(\mathcal{A})=N_{\sym(G)}(\aut(\mathcal{A})).
\end{equation}

For every $K \leq \aut(G)$, the partition of $G$ into the orbits of $K$ defines the $S$-ring $\mathcal{A}$ over~$G$. In this case, $\mathcal{A}$ is called \emph{cyclotomic} and denoted by $\cyc(K,G)$. Clearly, $K\leq \aut(\mathcal{A})_e$. If $\mathcal{A}$ is cyclotomic, then  $\mathcal{A}$ is schurian. 

Let $\mathcal{A}_1$ and $\mathcal{A}_2$ be $S$-rings, $\mathcal{X}_1=\mathcal{X}(\mathcal{A}_1)$, and $\mathcal{X}_2=\mathcal{X}(\mathcal{A}_2)$. Every algebraic isomorphism from $\mathcal{X}_1$ to $\mathcal{X}_2$ induces a bijection $\varphi$ from $\mathcal{S}(\mathcal{A}_1)$ to $\mathcal{S}(\mathcal{A}_2)$ such that 
$$c_{\varphi(X)\varphi(Y)}^{\varphi(Z)}=c_{XY}^Z$$ 
for all $X,Y,Z\in \mathcal{S}(\mathcal{A})$. For the convenience, further we will refer to the above bijection as an \emph{algebraic isomorphism} from $\mathcal{A}_1$ to $\mathcal{A}_2$. Every such algebraic isomorphism can be extended in a natural way to a bijection between $\mathcal{A}_1$- and $\mathcal{A}_2$-sets. The group of all algebraic isomorphisms from $\mathcal{A}$ to itself (algebraic automorphisms) and its subgroup consisting of all algebraic automorphisms inducing by the elements of $\iso(\mathcal{A})$ are denoted by $\aut_{\alg}(\mathcal{A})$ and $\aut^{\ind}_{\alg}(\mathcal{A})$, respectively. The groups $\aut_{\alg}(\mathcal{A})$ and $\aut_{\alg}(\mathcal{X}(\mathcal{A}))$ ($\aut^{\ind}_{\alg}(\mathcal{A})$ and $\aut^{\ind}_{\alg}(\mathcal{X}(\mathcal{A}))$, respectively) are equivalent as permutation groups.

It is easy to see that $f\in \iso(\mathcal{A})$ induces the trivial algebraic automorphism of $\mathcal{A}$ which fixes every basic set if and only if $f\in \aut(\mathcal{A})$ and hence $f_1,f_2\in \iso(\mathcal{A})$ induce the same algebraic isomorphism if and only if $f_2f_1^{-1}\in \aut(\mathcal{A})$. Therefore
\begin{equation}\label{algind}
|\aut^{\ind}_{\alg}(\mathcal{A})|=\frac{|\iso(\mathcal{A})|}{|\aut(\mathcal{A})|}.
\end{equation}

\subsection{WL-closure and WL-dimension}

The material of this subsection is taken from~\cite{BPR} (see also~\cite[Section~2.6]{CP}). The \emph{WL-closure} $\WL(\Gamma)$ of a digraph $\Gamma=(\Omega,D)$ is defined to be the smallest coherent configuration $\mathcal{X}=(\Omega,S)$ such that $D$ is a relation of $\mathcal{X}$. The tensor of the intersection numbers of $\WL(\Gamma)$ is an output of the Weisfeiler-Leman algorithm whose input graph is $\Gamma$. One can check that $\aut(\WL(\Gamma))=\aut(\Gamma)$. So if $\Gamma=\cay(G,X)$ is a Cayley digraph, then $\aut(\WL(\Gamma))\geq G_\r$ and hence $\WL(\Gamma)$ is a Cayley scheme over $G$. For the convenience, further we will refer to the $S$-ring $\mathcal{A}$ such that $\mathcal{X}(\mathcal{A})=\WL(\Gamma)$ as the WL-closure of the Cayley digraph $\Gamma$. Due to the definitions, $\mathcal{A}$ is the smallest (with respect to inclusion) $S$-ring $\mathcal{A}$ over $G$ such that $X$ is an $\mathcal{A}$-set. 

\begin{lemm}\cite[Theorem~2.6.4]{CP}\label{wliso}
Let $\Gamma_1$ and $\Gamma_2$ be digraphs and $f$ an isomorphism from $\Gamma_1$ to $\Gamma_2$. Then $f$ is an isomorphism from $\WL(\Gamma_1)$ to $\WL(\Gamma_2)$.
\end{lemm}

The next lemma can be deduced straightforwardly by the induction on the number of iterations of the Weisfeiler-Leman algorithm.

\begin{lemm}\label{wlprop}
Let $\Gamma_1=(\Omega_1,D_1)$ and $\Gamma_2=(\Omega_2,D_2)$ be digraphs and $\varphi$ an algebraic isomorphism from $\WL(\Gamma_1)$ to $\WL(\Gamma_2)$ such that $\varphi(D_1)=D_2$. Then $\Gamma_1$ and $\Gamma_2$ are $\WL$-equivalent.
\end{lemm}

The lemma below is a corollary of~\cite[Theorem~2.5]{FKV} and can be found, e.g., in~\cite[Lemma~3.1]{BPR}. 

\begin{lemm}\label{wldim}
Let $\Gamma$ be a digraph with vertex set $\Omega$ and $\mathcal{X}=\WL(\Gamma)$. Then the following statements hold.

$(1)$ $\dimwl(\Gamma)\leq 2$ if and only if $\mathcal{X}$ is separable.

$(2)$ If $\mathcal{X}_{\alpha}$ is separable for some $\alpha\in \Omega$, then $\dimwl(\Gamma)\leq 3$.
\end{lemm}

\section{An $S$-ring over a Heisenberg group}

In this section, we recall a construction of a cyclotomic $S$-ring from~\cite{MR}, where all the details can be found. Let $q$ be an odd prime power and $\mathbb{F}_q$ a finite field of order~$q$. Further till the end of the paper, 
$$G=H_3(q)=\left\{ 
\left(\begin{smallmatrix}
1 & x & z\\
0 & 1 & y\\
0 & 0 & 1
\end{smallmatrix}\right):~x,y,z\in \mathbb{F}_q \right\}$$
is a Heisenberg group of dimension~$3$ over $\mathbb{F}_q$. Put
$$Z=Z(G)=\left\{ 
\left(\begin{smallmatrix}
1 & 0 & z\\
0 & 1 & 0\\
0 & 0 & 1
\end{smallmatrix}\right):~z\in \mathbb{F}_q \right\}.$$
Clearly, $|G|=q^3$ and $|Z|=q$.

Let $\varepsilon\in \mathbb{F}_q$ be a nonsquare and 
$$\mathcal{M}=\mathcal{M}(\varepsilon)=\left\{ 
\left(\begin{smallmatrix}
\alpha & \beta \\
\varepsilon \beta & \alpha \\
\end{smallmatrix}\right):~\alpha,\beta\in \mathbb{F}_q,~(\alpha,\beta)\neq (0,0) \right\}.$$ 

Due to~\cite[Chapter~2.5]{LN}, the group $\mathcal{M}$ (with the standard matrix multiplication) is a subgroup of $\GL_2(q)$ isomorphic to the multiplicative group $\mathbb{F}_{q^2}^*$. In particular, $|\mathcal{M}|=q^2-1$ and $\mathcal{M}$ is cyclic. Given $M=\left(\begin{smallmatrix}
\alpha & \beta \\
\varepsilon \beta & \alpha \\
\end{smallmatrix}\right)\in \mathcal{M}$, let us define $\rho=\rho(M):G\rightarrow G$ as follows:
$$\rho\left(\left(\begin{smallmatrix}
1 & x & z\\
0 & 1 & y\\
0 & 0 & 1
\end{smallmatrix}\right)\right)=
\left(\begin{smallmatrix}
1 & \alpha x+\varepsilon \beta y & F_{\alpha,\beta}(x,y,z)\\
0 & 1 & \beta x+ \alpha y\\
0 & 0 & 1
\end{smallmatrix}\right),$$
where $F_{\alpha,\beta}(x,y,z)=\alpha\beta(\frac{x^2}{2}+\varepsilon\frac{y^2}{2})+\varepsilon \beta^2 xy +(\alpha^2-\varepsilon \beta^2)z$. The mapping $M\mapsto \rho(M)$ is a monomorphism from $\mathcal{M}$ to $\aut(G)$ (see~\cite{MR}). Let $K=\{\rho(M):~M\in \mathcal{M}\}$. The group $K$ is a cyclic group of order~$q^2-1$. Put $\mathcal{A}=\cyc(K,G)$. The basic sets of $\mathcal{A}$ are the following:
$$\{e\},~Z^\#,~Y_i,~i\in \mathbb{F}_q,$$
where 
$$Y_i=\left\{\left(\begin{smallmatrix}
1 & \alpha & \gamma_i(\alpha,\beta)\\
0 & 1 & \beta\\
0 & 0 & 1
\end{smallmatrix}\right),~\alpha,\beta\in \mathbb{F}_q,~(\alpha,\beta)\neq (0,0)\right\}$$ 
for $\gamma_i(\alpha,\beta)=\frac{\alpha\beta}{2}+(\alpha^2-\varepsilon \beta^2)i$. One can see that $|Y_i|=q^2-1$ and $Y_i^{(-1)}=Y_{-i}$ for every $i\in \mathbb{F}_q$. 

Given $i\in \mathbb{F}_q$, put $X_i=Y_i\cup\{e\}$. The next lemma immediately follows from~\cite[Corollary~6.4]{MR}.

\begin{lemm}\label{transver}
For every $i\in \mathbb{F}_q$, 
$$\underline{X_i}\cdot \underline{X_i}^{(-1)}=q^2e+q(\underline{G}-\underline{Z}),$$
i.e. $X_i$ is a DDS with parameters~$(q^2,q,q^2,0,q)$ which is a transversal for $Z$ in $G$.
\end{lemm}

Given $i,j\in \mathbb{F}_q^\infty=\mathbb{F}_q\cup \{\infty\}$, put 
$$\chi(i)=
\begin{cases}
-i,~i\neq \infty,\\
\infty,~i=\infty,\\
\end{cases}$$ 
and
$$\psi(i,j)=
\begin{cases}
\frac{ij+\delta}{i+j},~i+j\neq 0,~i,j\neq \infty,\\
\infty,~j=-i,\\
i,~i\neq\infty,~j=\infty,\\
j,~i=\infty,~j\neq \infty,\\
\end{cases}$$ 
where $\delta=\frac{\varepsilon}{16}$.

\begin{lemm}\cite[Lemma~6.7]{MR}\label{cyclgroup}
The set $\mathbb{F}_q^\infty$ equipped with the binary operation $\psi$ forms a cyclic group of order~$q+1$, where $\infty$ is the identity element and $\chi(i)$ is the element inverse to~$i\in \mathbb{F}_q^\infty$. 
\end{lemm}

The lemma below collects an information on the structure constants of $\mathcal{A}$ from~\cite[Eqs.~(9),(10),(14),(16),(17)]{MR} which we will need further.

\begin{lemm}\label{consts}
Given $i,j,k\in \mathbb{F}_q$, 
$$c_{Y_iY_j}^{Z^\#}=
\begin{cases}
0,~j=-i,\\
q+1,~j\neq-i,\\
\end{cases}
$$
$$c_{Y_iY_{-i}}^{Y_k}=
\begin{cases}
q,~k\notin\{i,-i\},\\
q-1,~k\in\{i,-i\}~\text{and}~i\neq 0,\\
q-2,~k=i=-i=0,
\end{cases}$$
and if $i+j\neq 0$, then
$$
c_{Y_iY_j}^{Y_k}=
\begin{cases}
q+1,~k\notin\{i,j,\psi(i,j)\},\\
1,~k=\psi(i,j),\\
q,~i\neq j,~\text{and}~k\in\{i,j\},\\
q-1,~k=i=j.
\end{cases}
$$
\end{lemm}

\begin{prop}\label{aut}
In the above notations, $\aut(\mathcal{A})=G_\r\rtimes K$.
\end{prop}

\begin{proof}
The $S$-ring $\mathcal{A}$ is cyclotomic and hence schurian. So each basic set of $\mathcal{A}$ is an orbit of the stabilizer $\aut(\mathcal{A})_e$. Obviously, $G_\r\rtimes K\leq \aut(\mathcal{A})$. One can see that 
$$|\aut(\mathcal{A})|=|G||\aut(\mathcal{A})_e|=|G||Y_0||\aut(\mathcal{A})_{e,y_0}|=q^3(q^2-1)|\aut(\mathcal{A})_{e,y_0}|=|G||K||\aut(\mathcal{A})_{e,y_0}|,$$
where $y_0=\left(\begin{smallmatrix}
1 & 1 & 0\\
0 & 1 & 0\\
0 & 0 & 1
\end{smallmatrix}\right)\in Y_0$. So to prove that $\aut(\mathcal{A})=G_\r\rtimes K$, it suffices to check that the stabilizer $\aut(\mathcal{A})_{e,y_0}$ is trivial. Let us prove that every $f\in \aut(\mathcal{A})_{e,y_0}$ is trivial.

\begin{lemm}\label{blocks0}
In the above notations, $f^{Z}=\id_{Z}$.
\end{lemm}

\begin{proof}
Let $z\in Z$ and $Y$ a basic set containing $zy_0^{-1}$. Note that $\{z\}=Yy_0\cap Z^\#$ because $Y\cup \{e\}$ is a transversal for $Z$ (Lemma~\ref{transver}). Therefore 
$$\{f(z)\}=f(Yy_0\cap Z^\#)=Yy_0\cap Z^\#=\{z\},$$
where the second equality holds by $f(y_0)=y_0$ and Eq.~\eqref{autring}. Thus, $f(z)=z$ for every $z\in Z$ and hence $f^{Z}=\id_{Z}$.
\end{proof}

\begin{lemm}\label{blocks}
In the above notations, if $f(y)=y$ for some $y\in G\setminus Z$, then $f(Zy)=Zy$ and $f^{Zy}=\id_{Zy}$. In particular, $f^{Zy_0}=\id_{Zy_0}$. 
\end{lemm}

\begin{proof}
Since $Z$ is an $\mathcal{A}$-subgroup, the set $Zy$ is a block of $\aut(\mathcal{A})$. So due to $f(y)=y$, we have $f(Zy)=Zy$. Each $X_i$ is a transversal for $Z$ by Lemma~\ref{transver}. Therefore all elements from $Zy$ belong to pairwise distinct basic sets. Together with $f(Zy)=Zy$, this implies that
$f^{Zy}=\id_{Zy}$.
\end{proof}

In view of Lemma~\ref{transver}, Lemma~\ref{blocks0}, and Lemma~\ref{blocks}, it remains to verify that $f$ fixes every $y\in Y_1\setminus Zy_0$. Since $X_0$ and $X_{-1}$ are transversals for $Z$ and $y\notin Zy_0$, there exist $i,j,k,l\in \{0,\ldots,q-1\}$ such that 
$$y\in T=Y_0z^i\cap Y_0y_0z^j\cap Y_{-1}z^k\cap Y_{-1}y_0z^l\cap Y_1.$$
If $|T|=1$ or, equivalently, $T=\{y\}$, then we are done. Indeed,
$$\{f(y)\}=f(T)=T=\{y\},$$
where the second equality holds by Eq.~\eqref{autring}, Lemma~\ref{blocks0}, and Lemma~\ref{blocks}. 

Further, we establish that $|T|=1$. Let $t\in T$. Then $t=\left(\begin{smallmatrix}
1 & \alpha & \gamma_1(\alpha,\beta)\\
0 & 1 & \beta\\
0 & 0 & 1
\end{smallmatrix}\right)\in Y_1$ for some $\alpha,\beta\in \mathbb{F}_q$. Let 
$$u_1=\left(\begin{smallmatrix}
1 & \alpha_1 & \gamma_0(\alpha_1,\beta_1)\\
0 & 1 & \beta_1\\
0 & 0 & 1
\end{smallmatrix}\right),~u_2=\left(\begin{smallmatrix}
1 & \alpha_2 & \gamma_0(\alpha_2,\beta_2)\\
0 & 1 & \beta_2\\
0 & 0 & 1
\end{smallmatrix}\right)\in Y_0$$ and 
$$u_3=\left(\begin{smallmatrix}
1 & \alpha_3 & \gamma_{-1}(\alpha_3,\beta_3)\\
0 & 1 & \beta_3\\
0 & 0 & 1
\end{smallmatrix}\right),~u_4=\left(\begin{smallmatrix}
1 & \alpha_4 & \gamma_{-1}(\alpha_4,\beta_4)\\
0 & 1 & \beta_4\\
0 & 0 & 1
\end{smallmatrix}\right)\in Y_{-1}$$ such that
$$t=u_1z^i=u_2y_0z^j=u_3z^k=u_4y_0z^l.$$
The latter equalities are equivalent to the system
\begin{equation}\label{intersectionbig}
\begin{cases}
\alpha=\alpha_1=\alpha_2+1=\alpha_3=\alpha_4+1,\\
\beta=\beta_1=\beta_2=\beta_3=\beta_4,\\
\gamma_1(\alpha,\beta)=\gamma_0(\alpha_1,\beta_1)+i=\gamma_0(\alpha_2,\beta_2)+j=\gamma_{-1}(\alpha_3,\beta_3)+k=\gamma_{-1}(\alpha_4,\beta_4)+l.\\
\end{cases}
\end{equation}
One can see that $\alpha_i,\beta_i$ can be expressed via $\alpha,\beta$ using the equations from the first and second lines of System~\eqref{intersectionbig} and these expressions can be substituted to the equations from the third line of System~\eqref{intersectionbig}. As a result, we obtain
$$\alpha^2-\varepsilon \beta^2=i=-\frac{\beta}{2}+j=-(\alpha^2-\varepsilon \beta^2)+k=-\frac{\beta}{2}-((\alpha-1)^2-\varepsilon \beta^2)=l.$$
In particular, $\beta=2(j-i)$ from the second equality and $\alpha=\frac{k+1+\frac{\beta}{2}}{2}$ from the fourth one. Therefore System~\eqref{intersectionbig} has at most one solution and hence $|T|\leq 1$. On the other hand, $y\in T$. Thus, $T=\{y\}$ as required.
\end{proof}

\begin{corl}\label{2min}
The $S$-ring $\mathcal{A}$ is normal and $2$-minimal.
\end{corl}

\begin{proof}
The normality immediately follows from Proposition~\ref{aut}. Observe that the basis relations $r(Y_i)$, $i\in \mathbb{F}_q$, of the Cayley scheme $\mathcal{X}(\mathcal{A})$ have size 
$$|G|n_{r(Y_i)}=|G||Y_i|=q^3(q^2-1)=|\aut(\mathcal{A})|.$$ 
Therefore there is no a proper subgroup of $\aut(\mathcal{A})$ with the same orbits on $G\times G$ as $\aut(\mathcal{A})$ and hence $\mathcal{A}$ is $2$-minimal.
\end{proof}

\begin{prop}\label{iso}
In the above notations, $|\aut^{\ind}_{\alg}(\mathcal{A})|\leq 2\log(q)$.
\end{prop}

\begin{proof}
Note that $|\aut^{\ind}_{\alg}(\mathcal{A})|=|\iso(\mathcal{A})|/|\aut(\mathcal{A})|$ by Eq.~\eqref{algind}. From Eq.~\eqref{schurnorm} it follows that $\iso(\mathcal{A})=N_{\sym(G)}(\aut(\mathcal{A}))$. By Proposition~\ref{aut}, we have $\aut(\mathcal{A})=G_\r\rtimes K$, where $|K|=q^2-1$. So the group $G_\r$ is a normal Sylow subgroup of $\aut(\mathcal{A})$. In particular, $G_\r$ is characteristic in $\aut(\mathcal{A})$. Since $\aut(\mathcal{A})$ is normal in $\iso(\mathcal{A})$, we conclude that $G_\r$ is normal in $\iso(\mathcal{A})$ and hence $\iso(\mathcal{A})\leq N_{\sym(G)}(G_\r)=\Hol(G)$. Thus,
$$\iso(\mathcal{A})\leq N_{\sym(G)}(\aut(\mathcal{A})) \cap \Hol(G)=N_{\Hol(G)}(G_\r\rtimes K).$$

The group $N_{\Hol(G)}(G_\r\rtimes K)$ coincides with $G_\r\rtimes N_{\aut(G)}(K)$. Indeed, $g_\r\tau\in \Hol(G)$, where $g\in G$ and $\tau\in \aut(G)$, normalizes the group $G_\r\rtimes K$ if and only if $G_\r\rtimes K^\tau=G_\r\rtimes K$. Clearly, if $\tau\in N_{\aut(G)}(K)$, then the latter equality holds. Suppose that $G_\r\rtimes K^\tau=G_\r\rtimes K$. Since $\GCD(|G|,|K|)=1$ and $G$ is solvable, the Schur-Zassenhaus theorem implies that $K$ and $K^\tau$ are conjugate in $G_\r\rtimes K$. So $K^{h_\r}=K^\tau$ for some $h\in G$. Then for every $\sigma \in K$, we have $h_\r\sigma h_\r^{-1}(e)=\sigma(h^{-1})h=e$ because $h_\r\sigma h_\r^{-1}\in K^\tau\leq \aut(G)$. This yields that $h$ is fixed by every element of $K$ and hence the basic set of $\mathcal{A}$ containing $h$ is of size~$1$. The unique basic set of $\mathcal{A}$ of size~$1$ is $\{e\}$. So $h=e$ and, consequently, $K^\tau=K$. Therefore $G_\r\rtimes K^\tau=G_\r\rtimes K$ if and only if $\tau\in N_{\aut(G)}(K)$ and hence $N_{\Hol(G)}(G_\r\rtimes K)=G_\r\rtimes N_{\aut(G)}(K)$. Thus,
$$\iso(\mathcal{A})\leq G_\r\rtimes N_{\aut(G)}(K).$$

Due to the latter equality and $\aut(\mathcal{A})=G_\r\rtimes K$, the statement of the proposition is a consequence of inequality
\begin{equation}\label{double}
|N_{\aut(G)}(K)|\leq 2\log(q)|K|
\end{equation}
which we are going to prove. Let $\pi$ be the homomorphism from $\aut(G)$ to $\aut(G/Z)$ which maps every $\sigma\in\aut(G)$ to the induced automorphism  of $G/Z$ acting by the rule $\sigma(Zg)=Z\sigma(g)$. The group $G/Z$ is isomorphic to the additive group of $(2l)$-dimensional vector space over $\mathbb{F}_p$, where $q=p^l$ for an odd prime $p$ and $l=\log(q)$, and hence 
$$\aut(G/Z)\cong \GL_{2l}(p).$$
The group $\pi(K)$ is a cyclic subgroup of $\aut(G/Z)$ of order~$q^2-1$. Moreover, $\pi(K)$ acts transitively on the set of nonidentity elements of $G/Z$ because each orbit of $K$ outside $Z$ together with the identity element is a transversal for $Z$ in $G$ (Lemma~\ref{transver}). Therefore $\pi(K)$ is a Singer subgroup of $\aut(G/Z)\cong \GL_{2l}(p)$. From~\cite[p.~187]{Hup} it follows that
\begin{equation}\label{singer}
|N_{\aut(G/Z)}(\pi(K))|=2l|\pi(K)|.
\end{equation}

Let $N_0=N_{\aut(G)}(K)\cap \ker(\pi)$ and $\theta\in N_0$. Then $\theta\sigma\theta^{-1}\in K$ for every $\sigma \in K$. Let $g\in G\setminus Z$. The elements $\theta\sigma\theta^{-1}(g)$ and $\sigma(g)$ lie in the same orbit of $K$ because $\theta\in N_{\aut(G)}(K)$ and in the same $Z$-coset, namely in $\sigma(Zg)$, because $\theta\in \ker(\pi)$. So $\theta\sigma\theta^{-1}(g)=\sigma(g)$ by Lemma~\ref{transver}. Since the above equality holds for all $g\in G\setminus Z$, we conclude that $\theta\sigma\theta^{-1}=\sigma$ for every $\sigma \in K$. Therefore $\theta\in C_{\aut(G)}(K)$ and hence $N_0\leq C_{\aut(G)}(K)$. The group $K$ is regular on each of its orbits outside $Z$. This yields that $C_{\aut(G)}(K)$ is semiregular on each orbit of $K$ outside $Z$. So $|C_{\aut(G)}(K)|\leq |K|$. One can see that $K$ is abelian and consequently $K\leq C_{\aut(G)}(K)$. Therefore $C_{\aut(G)}(K)=K$. The latter implies that $N_0\leq K_0$, where $K_0=K\cap \ker(\pi)$. However, each orbit of $K$ outside $Z$ together with the identity element is transversal for $Z$ in $G$ which implies that $K_0$ is trivial. Thus,
\begin{equation}\label{ker} 
|N_0|=|K_0|=1.
\end{equation}
Now one can estimate $|N_{\aut(G)}(K)|$ as follows:
$$|N_{\aut(G)}(K)|=|N_{\aut(G)}(K)/N_0|=|\pi(N_{\aut(G)}(K))|\leq |N_{\aut(G/Z)}(\pi(K))|=2l|\pi(K)|=2l|K|,$$
where the first and fifth equalities hold by Eq.~\eqref{ker}, whereas the fourth equality holds by Eq.~\eqref{singer}. Thus, Eq.~\eqref{double} holds and we are done. 
\end{proof}

Given $\tau\in \aut(\mathbb{F}_q^\infty)$, let us define a bijection $\widehat{\tau}$ on the set $\mathcal{S}(\mathcal{A})$ as follows:
$$\widehat{\tau}(\{e\})=\{e\},~\widehat{\tau}(Y_k)=Y_{\tau(k)},~k\in \mathbb{F}_q^\infty,$$
where $Y_{\infty}=Z^\#$.

\begin{prop}\label{autalg}
The following statements hold.
\begin{enumerate}

\tm{1} The mapping $\tau \mapsto \widehat{\tau}$ is a monomorphism from $\aut(\mathbb{F}_q^\infty)$ to $\aut_{\alg}(\mathcal{A})$.

\tm{2} $|\aut_{\alg}(\mathcal{A})|\geq\phi(q+1)$.

\tm{3} Given generators $i,j$ of $\mathbb{F}_q^\infty\cong C_{q+1}$, there is an algebraic automorphism of $\mathcal{A}$ which maps $Y_i$ to $Y_j$.

\end{enumerate} 
\end{prop}

\begin{proof}
Statement~$(1)$ immediately follows from the the definition of $\widehat{\tau}$ and inclusion $\widehat{\tau}\in \aut_{\alg}(\mathcal{A})$ which can be verified straightforwardly using Eq.~\eqref{triangle} and Lemma~\ref{consts}. For example,  
$$c_{\widehat{\tau}(Y_i)\widehat{\tau}(Y_j)}^{\widehat{\tau}(Y_{\psi(i,j)})}=c_{Y_{\tau(i)}Y_{\tau(j)}}^{Y_{\tau(\psi(i,j))}}=c_{Y_{\tau(i)}Y_{\tau(j)}}^{Y_{\psi(\tau(i),\tau(j))}},$$
where the latter equality holds because $\tau\in \aut(\mathbb{F}_q^\infty)$. Statement~$(2)$ holds by Statement~$(1)$ and Lemma~\ref{cyclgroup}. Since $i$ and $j$ are generators of $\mathbb{F}_q^\infty$, there is $\tau\in \aut(\mathbb{F}_q^\infty)$ such that $\tau(i)=j$. So $\widehat{\tau}(Y_i)=Y_{\tau(i)}=Y_j$ and Statement~$(3)$ holds by Statement~$(1)$.
\end{proof}

\begin{corl}\label{nonsep}
The $S$-ring $\mathcal{A}$ is separable only if $q\in\{3,5,9\}$.
\end{corl}

\begin{proof}
Due to Proposition~\ref{iso} and Proposition~\ref{autalg}(2), the $S$-ring $\mathcal{A}$ is separable only if $\phi(q+1)\leq 2\log(q)$ which is true if and only if $q\in\{3,5,9\}$.
\end{proof}

\section{Proof of Theorem~\ref{main}}

Throughout this section, we use the notations from the previous one. Given $i\in \mathbb{F}_q$, put $\Gamma_i=\cay(G,X_i)$. By Lemma~\ref{transver}, $\Gamma_i$ is a divisible design digraph with parameters~$(q^3,q^2,0,q,q^2,q)$ whose classes are $Z$-cosets and in which each vertex dominates exactly one vertex from each class. Denote the set of all generators of $\mathbb{F}_q^\infty\cong C_{q+1}$ by $I$. Clearly, $|I|=\phi(q+1)$.

\begin{lemm}\label{wlclosure}
In the above notations, $\WL(\Gamma_i)=\mathcal{A}$ for every $i\in I$.
\end{lemm}

\begin{proof}
Put $\mathcal{B}=\WL(\Gamma_i)$. The $S$-ring $\mathcal{B}$ is the smallest $S$-ring over $G$ for which $Y_i$ is a $\mathcal{B}$-set. So $\mathcal{B}\leq \mathcal{A}$. Let us check the reverse inclusion. Put 
$$a_1=i~\text{and}~a_k=\psi(a_{k-1},i)~\text{for}~k\geq 2.$$
Since $i$ is a generator of $\mathbb{F}_q^\infty$, we conclude that $a_{q+1}=\infty$ and
\begin{equation}\label{generat} 
\{a_1,\ldots,a_{q+1}\}=\mathbb{F}_q^\infty. 
\end{equation}

Further we are going to prove that $Y_{a_k}$ is a $\mathcal{B}$-set for every $k\in\{1,\ldots,\frac{q+1}{2}\}$. We proceed by induction on~$k$. The base of induction for $k=1$ follows from the definition of $\mathcal{B}$. Observe that $a_k\neq -i$ for $k\in\{1,\ldots,\frac{q+1}{2}\}$ because $-i=\chi(i)$ and hence $-i=a_q$. Together with Lemma~\ref{consts} this implies that
$$\underline{Y_{a_k}}\cdot \underline{Y_i}=\underline{Y_{a_{k+1}}}+q(\underline{Y_{a_k}}+\underline{Y_i})+(q+1)(\underline{G}^\#-\underline{Y_{a_{k+1}}}-\underline{Y_{a_k}}-\underline{Y_i}).$$
By the induction hypothesis, the left-hand side of the above equality belongs to $\mathcal{B}$. So $\underline{Y_{a_{k+1}}}\in \mathcal{B}$ by Lemma~\ref{sw} and hence $Y_{a_{k+1}}$ is a $\mathcal{B}$-set. For every $k\in \{\frac{q+1}{2},\ldots,q\}$, there is $l\in \{1,\ldots,\frac{q+1}{2}\}$ such that $Y_{a_k}=Y_{\chi(a_l)}=Y_{-a_l}=Y_{a_l}^{(-1)}$. Therefore $Y_{a_{k}}$ is a $\mathcal{B}$-set for every $k\in \{\frac{q+1}{2},\ldots,q\}$. The above discussion together with Eq.~\eqref{generat} yields that $Y_i$ is a $\mathcal{B}$-set for every $i\in \mathbb{F}_q$. Finally, 
$$Z^\#=G^\#\setminus \bigcup \limits_{i\in \mathbb{F}_q} Y_i$$
is a $\mathcal{B}$-set. Thus, every basic set of $\mathcal{A}$ is a $\mathcal{B}$-set. This implies that $\mathcal{B}\geq \mathcal{A}$ and hence $\mathcal{B}=\mathcal{A}$ as desired.
\end{proof}

\begin{corl}\label{wlgamma}
For all $i,j\in I$, the digraphs $\Gamma_i$ and $\Gamma_j$ are normal, arc-transitive, and $\WL$-equivalent.
\end{corl}

\begin{proof}
Due Lemma~\ref{wlclosure}, we have $\WL(\Gamma_i)=\WL(\Gamma_j)=\mathcal{A}$ and hence $\aut(\Gamma_i)=\aut(\Gamma_j)=\aut(\mathcal{A})$. So $\aut(\Gamma_i)=\aut(\Gamma_j)=G_\r\rtimes K$ by Proposition~\ref{aut}. This implies that $\Gamma_i$ and $\Gamma_j$ are normal. The group $K=\aut(\Gamma_i)_e=\aut(\Gamma_j)_e$ acts transitively on $Y_i$ and $Y_j$ which are neighborhoods of~$e$ in $\Gamma_i$ and $\Gamma_j$, respectively. Since $\Gamma_i$ and $\Gamma_j$ are vertex-transitive, this yields that $\Gamma_i$ and $\Gamma_j$ are arc-transitive. Finally, $\Gamma_i$ and $\Gamma_j$ are $\WL$-equivalent by Lemma~\ref{wlprop}, Proposition~\ref{autalg}(3), and Proposition~\ref{wlclosure}.
\end{proof}

\begin{lemm}\label{nonisogamma}
There are at least $\left[\frac{\phi(q+1)}{2\log(q)}\right]$ pairwise nonisomorphic digraphs among $\Gamma_i$, $i\in I$.
\end{lemm}

\begin{proof}
The relation ``to be isomorphic'' defines an equivalence $E$ on the set $\{\Gamma_i:~i\in I\}$. It suffices to verify that there are at least $\left[\frac{\phi(q+1)}{2\log(q)}\right]$ equivalence classes of $E$. Suppose that $\Gamma_i$ and $\Gamma_j$ are isomorphic for some $i,j\in I$ and $f$ is an isomorphism from $\Gamma_i$ to $\Gamma_j$. Then $f\in \iso(\mathcal{A})$ by Lemma~\ref{wliso} and Lemma~\ref{wlclosure}. Clearly, $f$ induces an algebraic automorphism $\varphi\in \aut^{\ind}_{\alg}(\mathcal{A})$ such that $\varphi(Y_i)=Y_j$. Proposition~\ref{iso} implies that $|\aut^{\ind}_{\alg}(\mathcal{A})|\leq 2\log(q)$. Thus, each class of $E$ has size at most $2\log(q)$ and hence the number of classes is at least $\left[\frac{\phi(q+1)}{2\log(q)}\right]$ as required.
\end{proof}

\begin{prop}\label{wldim3}
Let $i\in I$. Then $\dimwl(\Gamma_i)\leq 3$. Moreover, $\dimwl(\Gamma_i)=3$ if $q>9$ and $\dimwl(\Gamma_i)=2$ if $q=3$. 
\end{prop}

\begin{proof}
At first, let us prove that $\dimwl(\Gamma_i)\leq 3$. Due to Lemma~\ref{wldim}(2) and Lemma~\ref{wlclosure}, it suffices to verify that the one-point extension $\mathcal{X}_e$ of the Cayley scheme $\mathcal{X}=\mathcal{X}(\mathcal{A})$ is separable. Since $\mathcal{A}$ and hence $\mathcal{X}$ are schurian, $F(\mathcal{X}_e)=\mathcal{S}(\mathcal{A})$ by Lemma~\ref{onepoint}(2). Given $j\in \mathbb{F}_q^*$, put $k=k(j)=\psi(j,0)$. Lemma~\ref{onepoint}(1) implies that 
$$r_j=r(Y_{k(j)})\cap (Y_0\times Y_j)$$
is a relation of $\mathcal{X}_e$. Let $y\in Y_0$. Clearly, $yr_j=Y_ky\cap Y_j$. One can see that $|Y_ky\cap Y_j|=c_{Y_{-k}Y_j}^{Y_0}$. In view of Eq.~\eqref{triangle} and Lemma~\ref{consts}, we have $c_{Y_{-k}Y_j}^{Y_0}=c_{Y_jY_0}^{Y_k}=1$. Therefore $r_j$ is exactly a basis relation of $\mathcal{X}_e$ and $n_{r_j}=1$. Since each $X_j$ is a transversal for $Z$ in $G$ (Lemma~\ref{transver}), $|Y_jy\cap Z^\#|=1$. So all basis relations of $\mathcal{X}_e$ inside $Y_0\times Z^\#$ are of valency~$1$. The above discussion yields that for every  $\Delta\in F(\mathcal{X}_e)$, there is a basis relation of valency~$1$ inside $Y_0\times \Delta$. Thus, $\mathcal{X}_e$ is separable whenever $(\mathcal{X}_e)_{Y_0}$ so is by Lemma~\ref{deletefiber}(2).

Let us verify that $(\mathcal{X}_e)_{Y_0}$ is separable. From Lemma~\ref{deletefiber}(1) it follows that $\aut((\mathcal{X}_e)_{Y_0})\cong \aut(\mathcal{X}_e)$. Eq.~\eqref{autpoint} implies that $\aut(\mathcal{X}_e)=\aut(\mathcal{X})_e$. By Proposition~\ref{aut}, we have $\aut(\mathcal{X})=\aut(\mathcal{A})=G_\r\rtimes K$ and hence $\aut(\mathcal{X})_e=K\cong C_{q^2-1}$. Thus,
$$\aut((\mathcal{X}_e)_{Y_0})\cong C_{q^2-1}.$$

Since $Y_0$ is an orbit of $K$ and $|Y_0|=q^2-1$, the group $\aut((\mathcal{X}_e)_{Y_0})$ is regular. Therefore $(\mathcal{X}_e)_{Y_0}$ is isomorphic to a Cayley scheme over a cyclic group with a regular automorphism group. From~\cite[Corollary~4.4.4]{CP} it follows that $(\mathcal{X}_e)_{Y_0}$ is schurian and hence regular. Thus, $(\mathcal{X}_e)_{Y_0}$ is separable by Lemma~\ref{regularsep} as required.

If $q>9$, then Corollary~\ref{wlgamma} and Lemma~\ref{nonisogamma} imply that there is at least one digraph nonisomorphic to $\Gamma_i$ which is $\WL$-equivalent to $\Gamma_i$. So $\dimwl(\Gamma_i)>2$ and hence $\dimwl(\Gamma_i)=3$. If $q=3$, then is can be verified using the list of small schemes~\cite{HM} that $\mathcal{X}$ is separable and hence $\dimwl(\Gamma_i)=2$ by Lemma~\ref{wldim}(1).
\end{proof}

\begin{prop}\label{desiso}
For all $i,j\in \mathbb{F}_q$, the divisible designs $\mathcal{D}(\Gamma_i)=\dev(X_i)$ and $\mathcal{D}(\Gamma_j)=\dev(X_j)$ are isomorphic.
\end{prop}

\begin{proof}
Recall that given $i\in \mathbb{F}_q$, $\dev(X_i)=(G,\mathcal{B}(X_i))$, where $\mathcal{B}(X_i)=\{X_ig:~g\in G\}$. We are going to check that $\dev(X_i)$ is isomorphic to $\dev(X_0)$ for every $i\in \mathbb{F}_q^*$ and by that to prove the lemma. To do this, it suffices to find two bijections $f,h:G\rightarrow G$ such that
\begin{equation}\label{crit}
g\in X_0g_0\Leftrightarrow f(g)\in X_ih(g_0)
\end{equation}
for all $g=\left(\begin{smallmatrix}
1 & \alpha & \gamma\\
0 & 1 & \beta\\
0 & 0 & 1
\end{smallmatrix}\right),g_0=\left(\begin{smallmatrix}
1 & \alpha_0 & \gamma_0\\
0 & 1 & \beta_0\\
0 & 0 & 1
\end{smallmatrix}\right)\in G$. One can see that $g\in X_0g_0$ if and only if $gg_0^{-1}\in X_0$. Due to the definition of $X_0$, the latter is equivalent to
\begin{equation}\label{descond} 
\gamma-\gamma_0=\frac{(\alpha-\alpha_0)(\beta+\beta_0)}{2}.
\end{equation}

Let $f:\left(\begin{smallmatrix}
1 & \alpha & \gamma\\
0 & 1 & \beta\\
0 & 0 & 1
\end{smallmatrix}\right) \mapsto  \left(\begin{smallmatrix}
1 & \alpha^\prime & \gamma^\prime\\
0 & 1 & \beta^\prime\\
0 & 0 & 1
\end{smallmatrix}\right)$
and 
$h:\left(\begin{smallmatrix}
1 & \alpha_0 & \gamma_0\\
0 & 1 & \beta_0\\
0 & 0 & 1
\end{smallmatrix}\right) \mapsto  \left(\begin{smallmatrix}
1 & \alpha_0^{\prime\prime} & \gamma^{\prime\prime}\\
0 & 1 & \beta_0^{\prime\prime}\\
0 & 0 & 1
\end{smallmatrix}\right)$ be mappings from $G$ to itself such that

\begin{equation}\label{mapf}
\begin{cases}
\alpha=\alpha^\prime,\\
\beta=\beta^\prime,\\
\gamma=\gamma^\prime-((\alpha^\prime)^2-\varepsilon (\beta^\prime)^2)i,
\end{cases}
\end{equation}

\begin{equation}\label{maph}
\begin{cases}
\alpha_0=\alpha_0^{\prime\prime}-\delta,\\
\beta_0=\beta_0^{\prime\prime}-\sigma,\\
\gamma_0=\gamma_0^{\prime\prime}+\frac{(\alpha_0^{\prime\prime}-\delta)(\beta_0^{\prime\prime}-\sigma)}{2}-\frac{\alpha_0^{\prime\prime}\beta_0^{\prime\prime}}{2}+((\alpha_0^{\prime\prime})^2-\varepsilon (\beta_0^{\prime\prime})^2)i,
\end{cases}
\end{equation}
where 
\begin{equation}\label{express}
\delta=4i\varepsilon \beta_0^{\prime\prime}~\text{and}~\sigma=4i\alpha_0^{\prime\prime}.
\end{equation}
At first, let us verify that $f$ and $h$ are bijections. Since the mapping $\gamma\mapsto \gamma^\prime$ is a translation by a function of $\alpha$ and $\beta$, we conclude that $f$ is a bijection. Due to Eqs.~\eqref{maph} and~\eqref{express}, we obtain
$$\left(\begin{matrix}
 \alpha_0\\
 \beta_0\\
\end{matrix}\right)=
 A\left(\begin{matrix}
 \alpha_0^{\prime\prime}\\
 \beta_0^{\prime\prime}\\
\end{matrix}\right),$$
where 
$$A=\left(\begin{matrix}
1 & -4\varepsilon i  \\
-4i & 1\\
\end{matrix}\right).$$
One can see that $\det(A)=1-16\varepsilon i^2\neq 0$ because $\varepsilon$ is nonsquare. Therefore the mapping $$\left(\begin{matrix}
 \alpha_0\\
 \beta_0\\
\end{matrix}\right) \mapsto
 \left(\begin{matrix}
 \alpha_0^{\prime\prime}\\
 \beta_0^{\prime\prime}\\
\end{matrix}\right)$$
is a bijection. The mapping $\gamma_0\mapsto \gamma_0^{\prime\prime}$ is a translation by a function of $\alpha_0$ and $\beta_0$ which yields that $h$ is a bijection.

Substituting expressions for $\alpha$, $\beta$, $\gamma$ via $\alpha^\prime$, $\beta^\prime$, $\gamma^\prime$ from Eq.~\eqref{mapf} and for $\alpha_0$, $\beta_0$, $\gamma_0$ via $\alpha_0^{\prime\prime}$, $\beta_0^{\prime\prime}$, $\gamma_0^{\prime\prime}$ from Eq.~\eqref{maph} to Eq.~\eqref{descond} and making computations, we obtain 
$$\gamma^\prime-\gamma_0^{\prime\prime}=\frac{(\alpha^\prime-\alpha_0^{\prime\prime})(\beta^\prime+\beta_0^{\prime\prime})}{2}+((\alpha^\prime-\alpha_0^{\prime\prime})^2-\varepsilon (\beta^\prime-\beta_0^{\prime\prime})^2)i.$$
The latter is equivalent to $f(g)h(g_0)^{-1}=\left(\begin{smallmatrix}
1 & \alpha^\prime & \gamma^\prime\\
0 & 1 & \beta^\prime\\
0 & 0 & 1
\end{smallmatrix}\right)\left(\begin{smallmatrix}
1 & \alpha_0^{\prime\prime} & \gamma^{\prime\prime}\\
0 & 1 & \beta_0^{\prime\prime}\\
0 & 0 & 1
\end{smallmatrix}\right)^{-1}\in X_i$. Thus, Eq.~\eqref{crit} holds and hence $\dev(X_0)$ and $\dev(X_i)$ are isomorphic.
\end{proof}


Theorem~\ref{main} follows from Corollary~\ref{wlgamma}, Lemma~\ref{nonisogamma}, Proposition~\ref{wldim3}, and Proposition~\ref{desiso}.


\begin{thebibliography}{list}


\bibitem{AJP} 
\emph{K.~T.~Arasu,~D.~Jungnickel,~A.~Pott}, Divisible difference sets with multiplier~$-1$, J. Algebra, \textbf{133} (1990), 35--62.




\bibitem{BJL}
\emph{T.~Beth, D.~Jungnickel, H.~Lenz}, Design Theory, 2nd edition, Cambridge University Press, Cambridge (1999).


\bibitem{BPR}
\emph{R.~Bildanov, V.~Panshin,~G.~Ryabov}, On WL-rank and WL-dimension of some Deza circulant graphs,  Graphs Combin., \textbf{37}, No.~6 (2021), 2397--2421.


\bibitem{BCN}
\emph{A.~Brouwer, A.~Cohen, A.~Neumaier}, Distance-regular graphs, Springer, Heidelberg (1989).




\bibitem{CP}
\emph{G.~Chen, I.~Ponomarenko}, Coherent configurations, Central China Normal University Press, Wuhan (2019).


\bibitem{CH}
\emph{D. Crnkovi\'{c}, W. H. Haemers}, Walk-regular divisible design graphs, Des. Codes Cryptogr., \textbf{72} (2014) 165--175.



\bibitem{CK}
\emph{D. Crnkovi\'{c}, H. Kharaghani}, Divisible design digraphs, in: Algebraic Design Theory and Hadamard Matrices, (C. J. Colbourn, Ed.), Springer Proc. Math. Stat., Springer, New York \textbf{133} (2015), 43--60.



\bibitem{CKS}
\emph{D. Crnkovi\'{c}, H. Kharaghani, A. $\check{S}$vob}, Divisible design Cayley digraphs, Discrete Math., \textbf{343}, No.~4 (2020), Article ID 111784.



\bibitem{CS}
\emph{D. Crnkovi\'{c}, A. $\check{S}$vob}, New constructions of divisible design Cayley graphs, Graphs Comb., \textbf{38} (2022) Article number 17.





\bibitem{EP}
\emph{S.~Evdokimov,~I.~Ponomarenko}, Characterization of cyclotomic schemes and normal Schur rings over a cyclic group, St. Petersburg Math. J., \textbf{14}, No.~2 (2003), 189--221.



\bibitem{FKV}
\emph{F.~Fuhlbr\"uck, J~ K\"obler, O.~Verbitsky}, Identiability of graphs with small color classes by the Weisfeiler-Leman algorithm, in: Proc. $37$th International Symposium on Theoretical Aspects of Computer Science, Dagst\"uhl Publishing, Germany (2020), 43:1--43:18.




\bibitem{GHKS}
\emph{S. Goryainov, W.H. Haemers, V.~Kabanov, L.~Shalaginov}, Deza graphs with parameters $(n,k,k-1,a)$ and $\beta=1$, J. Comb. Des., \textbf{27} (2019) 188--202. 


\bibitem{Grohe} 
\emph{M.~Gr\"ohe}, Descriptive complexity, canonisation, and definable graph structure theory, Cambridge University Press, Cambridge (2017).


\bibitem{HKM}
\emph{W. H. Haemers, H. Kharaghani, M. A. Meulenberg}, Divisible design graphs,  J. Combin. Theory, Ser. A, \textbf{118} (2011) 978--992.


\bibitem{HM}
\emph{A.~Hanaki, I.~Miyamoto}, Classification of association schemes with small number of vertices. http://math.shinshu-u.ac.jp/*hanaki/as/ (2016).


\bibitem{Hup}
\emph{B.~Huppert}, Endliche Gruppen, Springer Berlin-Heidelberg-New York, \textbf{1} (1967).


\bibitem{Kabanov}
\emph{V.~Kabanov}, New versions of the Wallis-Fon-Der-Flaass construction to create divisible design graphs, Discrete Math., \textbf{345}, No.~11 (2022) Article ID 113054.



\bibitem{KS}
\emph{V.~Kabanov,~L.~Shalaginov}, On divisible design Cayley graphs,  Art Discrete Appl. Math., \textbf{4}, No.~2  (2021), 1--9.





\bibitem{LN}
\emph{R.~Lidl,~H.~Niederreiter}, Finite fields, Cambridge University Press, Second edition (1997).


\bibitem{MR}
\emph{M.~Muzychuk,~G.~Ryabov}, Constructing linked systems of relative difference sets via Schur rings, Des. Codes Cryptogr., \textbf{92} (2024), 2615--2637.


\bibitem{PS}
\emph{D. Panasenko, L. Shalaginov}, Classification of divisible design graphs with at most 39 vertices, J. Comb. Des., \textbf{30}, No.~4 (2022), 205--219.





\bibitem{RS}
\emph{J.~B.~Rosser,~L.~Schoenfeld}, Approximate formulas for some functions of prime numbers, Illinois J. Math., \textbf{6}, No.~1 (1962), 64--94.



\bibitem{Ry} 
\emph{G.~Ryabov}, On separability of Schur rings over abelian $p$-groups, Algebra Log., \textbf{57}, No.~1 (2018), 49--68.


\bibitem{ScSp}
\emph{R.-H. Schulz, A. G. Spera}, Divisible designs and groups, Geom. Dedicata, \textbf{44} (1992), 147--157.


\bibitem{Schur}
\emph{I.~Schur}, Zur theorie der einfach transitiven Permutationgruppen, S.-B. Preus Akad. Wiss. Phys.-Math. Kl., \textbf{18}, No.~20 (1933), 598--623.


\bibitem{Weis}
\emph{B. Weisfeiler (ed.)}, On construction and identification of graphs, Springer-Verlag, Berlin-New York (1976).




\bibitem{WeisL}
\emph{B.~Weisfeiler,~A.~Leman}, Reduction of a graph to a canonical form and an algebra which appears in the process, NTI, \textbf{2}, No.~9 (1968), 12--16.



\bibitem{Wi}
\emph{H.~Wielandt}, Finite permutation groups, Academic Press, New York - London (1964).

\end{thebibliography}
\end{document}